\documentclass{article}
\usepackage{amssymb,amsmath,amsthm,enumerate}
\makeatletter

\@addtoreset{equation}{section}
\makeatother
\newcommand{\lip}{\text{Lip}}

\renewcommand{\leq}{\leqslant}
\renewcommand{\geq}{\geqslant}

\providecommand{\R}{\mathbb{R}}
\providecommand{\C}{\mathbb{C}}
\providecommand{\T}{\mathbb{T}}
\providecommand{\N}{\mathbb{N}}
\providecommand{\Z}{\mathbb{Z}}
\providecommand{\V}{\mathcal{V}}

\newtheorem{Theorem}{Theorem}
\newtheorem{Definition}{Definition}
\newtheorem{Corollary}{Corollary}
\newtheorem{Proposition}{Proposition}
\newtheorem{Lemma}{Lemma}
\newtheorem{Remark}{Remark}

\date{\today}
\title{Smoothness of the flow map for low-regularity solutions of the Camassa-Holm equations}
\author{Olivier Glass\footnote{CEREMADE, UMR CNRS 7534,
Universit\'e Paris-Dauphine, 
Place du Mar\'echal de Lattre de Tassigny,
75775 Paris Cedex 16,
FRANCE
},
Franck Sueur\footnote{CNRS, UMR 7598, Laboratoire Jacques-Louis Lions, F-75005, Paris, France}
\footnote{UPMC Univ Paris 06, UMR 7598, Laboratoire Jacques-Louis Lions, F-75005, Paris, France}}
\begin{document}
\maketitle
\begin{abstract}
It was recently proven by De Lellis, Kappeler, and Topalov in \cite{suisses} that  the  periodic  Cauchy problem for the Camassa-Holm  equations is locally well-posed in the space $ \lip (\T)$  endowed with the topology of  $H^1 (\T)$. We prove here that the Lagrangian  flow of these solutions are   analytic with respect to time and smooth with respect to the initial data. \par
These results can be adapted to  the  higher-order Camassa-Holm equations describing the exponential curves of the manifold of orientation preserving diffeomorphisms of $\T$ using the Riemannian structure induced by the Sobolev inner product $H^l (\T)$, for $l \in \N$, $l\geq 2$ (the classical  Camassa-Holm equation  corresponds to the case $l=1$):  the  periodic  Cauchy problem   is locally well-posed in the space
$ W^{2l-1,\infty} (\T)$ endowed with the topology of  $H^{2l-1} (\T)$ and the Lagrangian  flows of these solutions are   analytic with respect to time with values in $ W^{2l-1,\infty} (\T)$  and smooth with respect to the initial data. \par
These results extend some earlier results which dealt with more regular solutions. In particular our results cover the case of peakons, up to the first collision.
\end{abstract}
\section{Introduction}
We consider the Cauchy problem for the Camassa-Holm  equations:
\begin{equation} \label{CH0}
\left\{ \begin{array}{ll}
 \partial_t u - \partial_{txx} u+ 3 u \partial_x u -2  \partial_x u   \partial_{xx} u - u \partial_{xxx} u = 0 , \\
 u(0,x)=u_0(x), 
\end{array} \right.
\end{equation}
where $x$ runs over the line $\R$ or the one-dimensional torus $\T$. \par
This equation has been derived independently by:
\begin{itemize}
\item  Fokas and Fuchssteiner \cite{FF} where it appears as a member of a whole family of bi-hamiltonian equations generated by the method of recursion operator,
\item Camassa and Holm in \cite{CH} as a model for the propagation of water waves in the shallow water regime, when the wavelength is considerably larger than the average water depth (then $u$ represents the velocity, and also, roughly, the
height of the water's free surface above a flat bottom see \cite{cl}), 
\item Dai \cite{dai} as  a model for the propagation of nonlinear waves in cylindrical hyperelastic rods (in this case the function $u$ represents the radial stretch).
\end{itemize}
This equation attracted a lot of attention for its rich structure: 
\begin{itemize}
\item it is formally integrable, in the sense that there is an associated Lax pair, 
\item its solitary waves are solitons, i.e. they retain their shape and speed after the interaction with waves of the same type,
\item the Camassa-Holm equation possesses not only solutions that are global in time but models also wave breaking  in finite time. Wave breaking is an important physical phenomenon which is not captured by the other standard shallow water equations, as for example the KdV equation, and therefore makes the Camassa-Holm equation particularly interesting in that context.
\end{itemize}
Applying the operator $(\partial_x^{2} - 1 )^{-1}$ to the equation  \eqref{CH0} we get the following non local form of the Camassa-Holm  equations:
\begin{equation} \label{CH}
\left\{ \begin{array}{ll}
 \partial_t u +  \partial_x (\frac{1}{2}u^{2} + P)  =  0 , \\
 - (\partial_x^{2} - 1 ) P = u^2 + \frac{1}{2} (\partial_x u)^2  ,
\end{array} \right.
\end{equation}
which have a sense in $\mathcal{S}'$ for all $u$ in  $L^{1}_{\text{loc}} (\R ; H^1)$, where $H^1$ denotes the usual Sobolev space of order $1$. This space $H^{1}$ may appear as a very natural space to study this equation since, formally, the solutions of  \eqref{CH} preserves the $H^{1}$ norm.
Actually one gets a better insight into this time-invariance of the $H^{1}$ norm by introducing the quantity
\begin{equation}  \label{m}
m := u -  \partial_{x}^2 u .
\end{equation}
This quantity $m$ plays a special role in the theory. It can be compared to the vorticity for the incompressible Euler equation.
Moreover, if the equation is set on $\R$, integrating by parts and using that the relation \eqref{m} can be inverted into the formula 
\begin{equation}  \label{mi}
u(x) =  \frac{1}{2} \int_{ \R} m(y) e^{-|x-y|} \, dy,
\end{equation}
we obtain that the square of the $H^1$ norm 
\begin{equation}  \label{pe}
\| u \|^{2}_{H^1 (\R) } = \frac{1}{2} \int_{ \R} \int_{ \R} m(y) m(x) e^{-|x-y|} \, dx \, dy 
\end{equation}
can be reinterpreted as an interaction energy. \par
\ \par
Yet the well-posedness of the equation  \eqref{CH} for initial data $u_{0}$ in $H^{1}$ is a difficult issue. 
The first results of well-posedness  were obtained for smoother data. 
For instance  Constantin and Escher \cite{ce} first established the well-posedness of  the equation  \eqref{CH0} for $u_{0}$ in the Sobolev space $H^{3} (\R)$, by  applying Kato's theory for a hyperbolic quasi-linear PDE satisfied by $m$. \par
It has been subsequently proved by various methods  that the problem  \eqref{CH} is locally well-posed in the Sobolev space $H^s (\R)$ with $s > 3/2$ in 
\cite{RB,LO,M,Dfirst,D}.
These results hold only for a finite time interval, as it may appear a point where the profile of $u$ steepens gradually and ultimately the slope becomes vertical. In the context of water waves, this corresponds to the breaking of a wave, see the paper of Constantin and Escher \cite{acta} (see also the recent paper \cite{Brandolese}). In \cite{D}, Danchin have extended well-posedness to the Besov space $B^\frac{3}{2}_{2,1} (\R) $. \par
In \cite{M} G.  Misiolek proved that \eqref{CH} is locally well-posed in the space $  C^1 (\T)$ of continuously differentiable functions of the one-dimensional torus $\T$. 
Finally  De Lellis, Kappeler, and Topalov \cite{suisses} recently proved that the equation \eqref{CH} set on the torus $\T$ is well-posed in $W^{s,p} (\T) \cap \lip  (\T) $, endowed with the topology of  $W^{s,p} (\T)$, provided that $1 \leq s < 2$ and $1  \leq p < + \infty$. 
We rephrase their result here, only in the case $p=2$,  for sake of simplicity.
\begin{Theorem}[\cite{suisses}] \label{start1}
Let  $u_0$ be in $ \lip (\T)$. Then there exists $T > 0$ and a unique solution 
\begin{equation*}
u \in C (  ( -T,T),  (\lip (\T) , \sigma_{H^1 (\T)} ))  \cap C^1 (  ( -T,T);  (L^{\infty} (\T) , \sigma_{L^{2} (\T) } )) 
\end{equation*}
of \eqref{CH}. 
Moreover for any $u_0$ in $ \lip (\T) $ there exists $T > 0$ and  a neighborhood $ \mathcal V$ of $u_0$ in  $(\lip (\T) , \sigma_{H^1 (\T)} )$ such that the map 
\begin{equation}  \label{DC}
\left\{ \begin{array}{ll}
 \mathcal V \rightarrow   C (  ( -T,T);  (\lip (\T) , \sigma_{H^1 (\T)} ))  \cap C^1 (  ( -T,T);  (L^{\infty} (\T) , \sigma_{L^{2} (\T) } )) ,\\
 u_0 \mapsto \text{ the corresponding solution } u \text{ of } \eqref{CH} 
\end{array} \right.
\end{equation}
is well-defined and continuous.
\end{Theorem}
In the statement above the notation $(\lip (\T) , \sigma_{H^1 (\T)} )$ and $(L^{\infty} (\T) , \sigma_{L^{2} (\T) } )$ means that the 
space $\lip (\T)$ and $ L^{\infty} (\T)$ are respectively endowed with the topology of  $H^1 (\T)$ and $L^{2} (\T) $. 
It is shown in \cite[Section 1.2]{suisses} that the solution $u$ given by Theorem \ref{start1} may not be continuous with respect to time in the space $\lip (\T)$, endowed with its usual strong topology. \par
The proof of Theorem \ref{start1} relies on the Lagrangian description of the Camassa-Holm equation and on the observation that the flow map $\xi$ defined from  $ ( -T,T) \times \T$ to $\T$  by
\begin{equation} \label{flow}
\xi(t,x)  =  x + \int^{t}_0 u (s,\xi(s,x)) \, ds 
\end{equation}
and the Lagrangian velocity $u(t,\xi(t,x))$ are more regular than the Eulerian velocity $u(t,x)$.
In particular $\xi$ is $C^{1}$ over $(-T,T) $ with values in $\lip(\T)$, endowed with its strong topology. \par
The aim of this paper is to prove that the flow map of these solutions actually benefits from extra smoothness properties. 
In particular we prove the following result concerning the smoothness in time of the flow map.
\begin{Theorem} \label{start2}
Assume that $u_0$ is in $\lip (\T)$.
Then, with the notations above, $ \xi$ belongs to the space $C^{\omega}((-T,T); \lip(\T))$ of the analytic functions from  $(-T,T)$ to $\lip(\T)$.
\end{Theorem}
As a corollary of Theorem \ref{start2} we infer that the flow is also smooth with respect to the initial data. 
\begin{Corollary} \label{start3}
With the previous notations, the map 
\begin{equation} \label{SC}
\left\{ \begin{array}{ll}
 \mathcal V \rightarrow   C^{\omega } (  ( -T,T);   \lip (\T))  ,\\
u_0 \mapsto \xi 
\end{array} \right.
\end{equation}
is of class $C^{\infty}$.
\end{Corollary}
Actually our results cover a more general case that we are now going to describe.
Let us first recall that the Camassa-Holm equation describes the exponential curves of the manifold of orientation preserving diffeomorphisms of $\T$ using the Riemannian structure induced by the Sobolev inner product $H^1 (\T)$ (see for instance \cite{Lenells} and references therein). \par
This can be seen as a counterpart of the celebrated papers \cite{Arnold}  and \cite{EbinMarsden}  where classical solutions of the 
incompressible Euler equations are interpreted as  geodesics of a Riemannian manifold of infinite dimension. \par
In this spirit some  higher-order Camassa-Holm equations can be considered  using the Sobolev space $H^l (\T)$, for $l \in \N$, $l\geq 2$, instead of $H^1 (\T)$, see for instance \cite{helvet,JNMP}. 
These equations read:
\begin{eqnarray} \label{CHorderl}
\partial_t u  + u \partial_x u = A_l^{-1} C_l (u) 
\end{eqnarray}
where 
\begin{eqnarray} \label{cla}
A_l := \sum_{j=0}^l (-1)^j \partial_x^{2j} , \quad
C_l (u) := -u A_l \partial_x u + A_l (u \partial_x u ) - 2 ( \partial_x u) A_l u .
\end{eqnarray}
For these equations we have the following  well-posedness  result which is a counterpart of Theorem \ref{start1}.
\begin{Theorem} \label{start1k}
Let $l \in \N$, $l\geq 1$.
Let  $u_0$ be in $ W^{2l-1,\infty} (\T)$. Then there exists $T > 0$ and only one solution 
\begin{equation*}
u \in C (( -T,T),  (W^{2l-1,\infty}(\T), \sigma_{H^{2l-1} (\T)} )) \cap C^1 (( -T,T), ( W^{2l-2,\infty} (\T), \sigma_{H^{2l-2} (\T) }))
\end{equation*}
of \eqref{CH}.  Moreover for any $u_0$ in $W^{l,\infty} (\T)$ there exists $T > 0$ and a neighborhood $\mathcal V$ of $u_0$ in $(W^{2l-1,\infty}(\T), \sigma_{H^{2l-1}(\T)})$ such that the map 
\begin{equation} \label{DCk}
\left\{ \begin{array}{ll}
 \mathcal V \rightarrow C(( -T,T), (W^{2l-1,\infty}(\T), \sigma_{H^{2l-1} (\T)} ))  \cap C^1 (( -T,T), (W^{2l-2,\infty} (\T), \sigma_{H^{2l-2} (\T) } )), \\
 u_0 \mapsto \text{ the corresponding solution } u \text{ of } \eqref{CHorderl} 
\end{array} \right.
\end{equation}
is well-defined and continuous.
\end{Theorem}
In the statement above the notation $(W^{2l-1,\infty} (\T) , \sigma_{H^{2l-1} (\T)} )$ and $( W^{2l-2,\infty} (\T) , \sigma_{H^{2l-2} (\T) } ))$ stands for the spaces $W^{2l-1,\infty} (\T)$ and $ W^{2l-2,\infty} (\T)$ endowed with the topologies of $H^{2l-1}(\T)$ and $H^{2l-2}(\T)$ respectively. \par
We will sketch the proof of this Theorem in Appendix for sake of completeness. 
Similarly to the case $l=1$ this proof also provides that  the flow
$\xi$ defined from  $ ( -T,T) \times \T$ to $\T$ by
\begin{equation*}
 \xi(t,x)  =  x + \int^{t}_0 u (s, \xi(s,x) ) ds 
\end{equation*}
is $C^{1}$ over $ ( -T,T) $ with values in $W^{2l-1,\infty} (\T)$. \par
The next result is the main one of the paper; it shows that  the  smoothness of the flow map is actually much better. 
\begin{Theorem} \label{start2k}
Assume that $u_0$ is in  $ W^{2l-1,\infty} (\T)$.
Then, with the notations above, the flow map $\xi$ is in the space $C^{\omega}((-T,T), W^{2l-1,\infty}(\T))$.
\end{Theorem}
Then Corollary  \ref{start3} extends as follows.
\begin{Corollary} \label{start3k}
With the previous notations,  the map 
$u_0 \in  \mathcal V \mapsto \xi  \in C^{\omega } (  ( -T,T);  W^{2l-1,\infty} (\T))  $
is $C^{\infty}$.
\end{Corollary}
Three remarks are in order.
\begin{itemize}
\item This extends some earlier results by \cite{helvet,JNMP,KLT} where the initial data  $u_0$ was assumed to be much smoother.  
\item Theorem \ref{start2k} can also be seen as a counterpart of the results of analyticity of the trajectories for the incompressible Euler equation, see \cite{Serfati3}, \cite{ogfstt} and the references therein. In particular the proof of Corollary \ref{start3k} follows the same lines as the proof of \cite[Corollary 1]{ogfstt} and will be therefore omitted here.
\item Let us stress that the results above are only local in time, including the existence of solutions  ``\`a la de~Lellis-Kappeler-Topalov''.  
We refer here to the papers \cite{MZ,MZZ} about the issue of the finite time blow up of  some classical solutions (with a slightly different regularity) of the higher-order Camassa-Holm equations. In the case of the Camassa-Holm equations, when $l=1$, the issue of wave breaking in finite time is a longstanding feature, intimately related to  the interaction of  the peakons. We will discuss this issue more closely in Section \ref{picon}.
\end{itemize}
Finally let us mention here the papers \cite{Glass} and  \cite{Perrollaz} which deal with the issues of control and of stabilization of the Camassa-Holm equation. 
\section{Proof of Theorem \ref{start2k}}
This section is devoted to the proof of Theorem \ref{start2k}.
\subsection{Preliminary material}
Let us start with a few remarks. \par
\ \par
\noindent
{\bf Rewriting the equation.} In the sequel we will simply denote by $A$ the operator $A_l$ defined in  \eqref{cla}. Following \cite{chk} the equation \eqref{CHorderl} also reads:
\begin{equation} \label{CHk}
\partial_t u  + u \partial_x u + \partial_x P =0 , \quad A P = \mathcal{F} [u] ,
\end{equation}
where $ \mathcal{F} [u]$ is a differential polynomial in $u$ of order $2l-1$ of the form 
\begin{equation} \label{CHkF}
\mathcal{F} [u] :=  \sum_{0 \leq m_{1} + m_{2} \leq 2l - 1 } \mathfrak{c}_{m_{1},m_{2}} \, (\partial_x^{m_{1}} u  ) ( \partial_x^{m_{2}} u ) ,
\end{equation}
where the $\mathfrak{c}_{m_{1}, m_{2}}$ are some real numbers. \par
Denoting the material derivative
\begin{equation*}
D := \partial_{t} + u \partial_{x},
\end{equation*}
the equation \eqref{CHorderl} takes the form
\begin{equation*} 
Du=- \partial_x P, \quad A P = \mathcal{F} [u] .
\end{equation*}
\ \par
\noindent
{\bf Elliptic operators.} Let us consider several operators associated to $A$. \par
We define 
\begin{equation} \label{deflambda}
\xi_{j} := e^{ \frac{i \pi j}{l+1}} , \quad {\Lambda}_{\pm} :=   -i \partial_x \pm \xi_{l}   \text{  and  } \tilde{\Lambda} _{\pm} := {\Lambda}_{\mp}  \prod_{j=1}^{l-1} \Big( ( -i  \partial_x  -  \xi_{j})  (  -i \partial_x  + \xi_{j}) \Big) .
\end{equation}
Then ${\Lambda}_{\pm}$ is under the form 
\begin{equation} \label{combibi}
\tilde{\Lambda} _{\pm} :=  \sum_{0 \leq  m \leq 2l-1  } d_{m} \partial_{x}^{m} ,
\end{equation}
where the $d_{m}$ are some complex numbers,  and
\begin{equation} \label{secon}
-2i \partial_x =  {\Lambda}_{+} +  {\Lambda}_{-}   \text{  and  }
{A}  =  {\Lambda}_{+} \,  \tilde{\Lambda}_{+} =  {\Lambda}_{-} \,   \tilde{\Lambda}_{-} .
\end{equation}
Now the operators $A$ and $\tilde{\Lambda}_{\pm}$ have the following property.
\begin{Lemma} \label{LemElliptic2l-1}
The operator $A$ is bicontinuous from $W^{2l,\infty} (\T)$ to $L^{\infty} (\T )$.
The operators $\tilde{\Lambda}_{\pm}$ are bicontinuous from $W^{2l-1,\infty} (\T)$ to $L^{\infty} (\T )$. 
\end{Lemma}
In particular this yields that when $u \in  W^{2l-1,\infty} (\T)$ then $ \mathcal{F} [u] $ is in $L^\infty (\T)$ and  $ \partial_x P $ is in $W^{2l-1,\infty} (\T)$. \par
\ \par
Lemma \ref{LemElliptic2l-1} is an immediate consequence from the following elementary one.
\begin{Lemma} \label{LemElliptic1}
Let $\xi \in \C \setminus 2 \pi \Z$  and $j \in \N$. Then the operator 
$-i \partial_x -  \xi$ is bicontinuous from $W^{j+1,\infty} (\T)$ to $W^{j,\infty} (\T)$.
\end{Lemma}
\begin{proof}[Proof of Lemma \ref{LemElliptic1}]
It is clear that  the operator $-i \partial_x - \xi$ is continuous from $W^{j+1,\infty} (\T)$ to $W^{j,\infty} (\T)$.
In order to prove that it is one-to-one, let us solve the equation in $f$:
\begin{eqnarray} \label{varc}
(-i \partial_x -  \xi) f = g ,
\end{eqnarray}
where $g \in W^{j,\infty} (\T)$ is given. 
The functions $f$ solution to \eqref{varc} are given by
\begin{eqnarray} \label{varcsol}
-i f(x) = C e^{i\xi x} +e^{i\xi x} \int_{0}^{x} g(x') e^{-i\xi x'} \, dx' ,
\end{eqnarray}
for some $C \in \R$.
Now to be one-periodic the function $f$ must satisfy $f(0)=f(1)$ so that
\begin{eqnarray*}
C (1-e^{i\xi }) =  e^{i\xi } \int_{0}^{1} g(x') e^{-i\xi x'} \, dx' ,
\end{eqnarray*}
which defines $C$ in a unique way since $\xi \in \C \setminus 2 \pi \Z$.
Now it is clear from the formula \eqref{varcsol} that the operator $(-i \partial_x - \xi)$ has a continuous inverse from $W^{j,\infty} (\T)$ to $W^{j+1,\infty} (\T)$.
\end{proof}
\ \par
\noindent
{\bf A formal identity.} We will use the following identity, which details the lack of commutation between  $D^{k}$ and $\partial_x^m$, for $m \in \N^*$. This approach is inspired by \cite{katoana}.
\begin{Lemma} \label{P1k}
For $k \in \N^*$, for  $m \in \N^*$, we have for smooth functions $u, \psi \in C^{\infty}((-T,T) \times \T;\R)$,
\begin{equation}\label{P1fNewk}
\partial_x^m D^k \psi = D^k  \partial_x^m  \psi + F^{k,m} [u,\psi ] \ \text{ where } \
 F^{k,m} [u,\psi] := \sum_{\gamma    \in \mathcal{B}_{k,m}  } c_{k,m} (\gamma   )  \,  f(\gamma)  [u,\psi],
\end{equation}
where
\begin{align*}
\mathcal{B}_{k,m} := \big\{ \gamma = (s, \alpha, \beta) \ \big/ \  
& s \in \N, \ \ 2 \leqslant s \leqslant k+1, \\
& \alpha = (\alpha_1, \ldots, \alpha_s) \in \N^s, \ \ | \alpha | = k+1 - s, \\
& \beta = (  \beta_1,\ldots,  \beta_s ) \in  (\N^*)^s, \ \ | \beta | = m + s - 1 \big\},
\end{align*}
\begin{equation*}
f( \gamma)  [u,\psi] : = \partial^{\beta_1}_x D^{\alpha_1} u \cdot \ldots \cdot \partial^{\beta_{s-1}}_x D^{\alpha_{s-1 } } u \cdot \partial^{\beta_{s}}_x D^{\alpha_s} \psi,
\end{equation*}
and where the $c_{k,m} (\gamma  )$ are integers satisfying 
\begin{equation} \label{EstCkm}
|c_{k,m} ( \gamma) | \leqslant {(2s)}^{2(m-1)} \frac{k! m!}{\alpha ! \beta !}.
\end{equation}
\end{Lemma}
Above, for a multi-index $\alpha$, we used the notations
\begin{equation*}
| \alpha | := \alpha_1 + \ldots+  \alpha_s \ \text{ and } \  \alpha ! :=  \alpha_1 ! \ldots \alpha_s !. 
\end{equation*}
\begin{Remark}
The writing of $f(\gamma)[u,\psi]$ is not unique in the sense that one can have $f(\gamma)[u,\psi]=f(\gamma')[u,\psi]$ with $\gamma \neq \gamma'$. However, it will be easier in the proof not to regroup identical terms and not to use the commutativity of the multiplication. We will only rely on Leibniz's rule and on
\begin{equation*}
\partial_{x} D = D \partial_{x} + (\partial_{x} u) \partial_{x}.
\end{equation*}
\end{Remark}
\begin{proof}[Proof of Lemma \ref{P1k}]
We proceed by iteration on $m$. \par
\noindent
{\bf 1.} The case $m=1$, which corresponds to the  commutation between  $D^{k}$ and $\partial_x$ is given by the following lemma.
Since it can be straightforwardly adapted from \cite[Proposition 6]{ogfstt}, its proof is omitted.
\begin{Lemma} \label{P1}
For $k \in \N^*$, we have in $\T$, for we have for smooth functions $u, \psi \in C^{\infty}((-T,T) \times \T;\R)$,
\begin{eqnarray*}
\partial_x D^k \psi = D^k  \partial_x  \psi + F^k [u,\psi ] \ \text{ with } \
F^k [u,\psi] := \sum_{\theta   \in \mathcal{A}_{k}  } c_k (\theta  )  \,  f(\theta)  [u,\psi], 
\end{eqnarray*}
where
\begin{equation*}
\mathcal{A}_{k} := \big\{ \theta  := (s, \alpha) \ \big/ \ 
s \in \N \ \text{ with } \ 2 \leqslant s  \leqslant k+1, \ \ 
\alpha = ( \alpha_1,\ldots, \alpha_s )  \in \N^s \ \text{ with } \ |\alpha| = k+1 - s \big\}, 
\end{equation*}
\begin{equation*}
f( \theta)  [u,\psi] : = \partial_x D^{\alpha_1} u \cdot \ldots \cdot \partial_x D^{\alpha_{s-1 } } u \cdot \partial_x D^{\alpha_s} \psi  ,
\end{equation*}
and where the $c_k (\theta )$ are integers satisfying
\begin{equation} \label{EstCk}
|c_k ( \theta) | \leqslant \frac{k ! }{\alpha !}.
\end{equation}
\end{Lemma}
Hence to conclude in the case $m=1$, it suffices to set $ F^{k,1} [u,\psi] := F^k [u,\psi]$ and to observe that $\gamma  := (s, \alpha , \beta)   \in \mathcal{B}_{k,1}$ means exactly that $\theta := (s, \alpha) $ belongs to $\mathcal{A}_{k}$ and $\beta = (1,\ldots,1)$. \par
\ \par
\noindent
{\bf 2.}
Now let us assume that  Lemma \ref{P1k} holds true up to order $m$ and let us prove that it also holds at  order $m+1$.
We will proceed in two steps. First we are going to prove an identity of the form \eqref{P1fNewk}, then we will estimate the coefficients involved in this expression. \par
\ \par
\noindent
{\bf a.} First, we observe that, using Lemma \ref{P1k} at order $m$  and Lemma \ref{P1} with $\partial_x^{m} \psi$ instead of $ \psi$, we get
\begin{eqnarray}
\nonumber
\partial_x^{m+1} D^k \psi &=& \partial_x \big( D^k  \partial_x^{m}  \psi + F^{k,m} [u,\psi ] \big), \\
\nonumber 
&= & D^k \partial_x^{m+1} \psi  + F^k [u,  \partial_x^{m} \psi ] +  \partial_x F^{k,m} [u,\psi ] .
\end{eqnarray}
Moreover
\begin{equation*}
F^k [u,  \partial_x^{m} \psi ] =  \sum_{\tilde{\theta} \in \mathcal{A}_{k}} c_k (\tilde{\theta})  \,  f(\tilde{\theta})  [u, \partial_x^{m}  \psi] 
\end{equation*}
and for $\tilde{\theta}=(\tilde{s},\tilde{\alpha})$,
\begin{eqnarray*}
f(\tilde{\theta})  [u, \partial_x^{m}  \psi] &=&  \partial_x D^{ \tilde{\alpha}_1 } u \cdot \ldots \cdot \partial_x D^{ \tilde{\alpha}_{\tilde{s}-1} } u \cdot \partial_x D^{ \tilde{\alpha}_{\tilde{s}} }  \partial_x^{m}  \psi  \\
& =& \partial_x D^{\tilde{\alpha}_1} u \cdot \ldots \cdot \partial_x D^{\tilde{\alpha}_{\tilde{s}-1} } u \cdot  \partial_x^{m+1}   D^{\tilde{\alpha}_s} \psi 
+  \partial_x D^{\tilde{\alpha}_1} u \cdot \ldots \cdot \partial_x D^{ \tilde{\alpha}_{\tilde{s}-1} } u \cdot  \partial_x  F^{\tilde{\alpha}_{\tilde{s}} ,m} [u,\psi] ,
\end{eqnarray*}
using again Lemma \ref{P1k} at order $m$ with $\tilde{\alpha}_{\tilde{s}}$ instead of $k$. \par
It follows that the commutator of $\partial_x^{m+1}$ and $D^k$ can be described as
\begin{align} 
\partial_x^{m+1} D^k \psi - D^k \partial_x^{m+1} \psi
\nonumber
&= \sum_{\tilde{\theta} \in \mathcal{A}_{k}}  c_k (\tilde{\theta})  \,  \partial_x D^{\tilde{\alpha}_1} u \cdot \ldots \cdot 
\partial_x D^{ \tilde{\alpha}_{\tilde{s}-1} } u \cdot  \partial_x^{m+1} D^{ \tilde{\alpha}_{\tilde{s}} } \psi  \\
\nonumber
&+ \sum_{\tilde{\theta}=(\tilde{s},\tilde{\alpha}) \in \mathcal{A}_{k}  } \sum_{\tilde{\gamma} \in {\mathcal B}_{\tilde{\alpha}_{\tilde{s}},m} }
c_k (\tilde{\theta}) \, c_{\tilde{\alpha}_{\tilde{s}},m}(\tilde{\gamma}) \, \partial_x D^{\tilde{\alpha}_1} u \cdot \ldots \cdot \partial_x D^{\tilde{\alpha}_{\tilde{s}-1 } } u \cdot  \partial_x  f(\tilde{\gamma}) [u,\psi] \\
\label{Comm}
&+ \sum_{\tilde{\gamma} \in \mathcal{B}_{k,m}  } c_k (\tilde{\gamma}) \, \partial_x f(\tilde{\gamma})[u,\psi].
\end{align}
We now discuss according to the three terms in the right hand side of \eqref{Comm}. \par
The first kind of terms in the right hand side of \eqref{Comm} have the desired form: they can written as $f(\gamma)  [u,\psi]$
with $\gamma := (s, \alpha , \beta)  $ 
where $s=\tilde{s}$, $\alpha := ( \tilde{\alpha}_1, \ldots, \tilde{\alpha}_{\tilde{s}} )$
and $\beta = ( 1 , \ldots ,1, m+1) \in (\N^*)^s $  satisfy  $| \alpha |    = k+1 - s $ and  $ |   \beta |    =  m+s $. \par
Therefore, it remains to prove the same for the second and third kind of terms in \eqref{Comm}.
Concerning the third kind of terms in \eqref{Comm}, using Leibniz's rule, we see that for $\tilde{\gamma} = (\tilde{s},\tilde{\alpha},\tilde{\beta}) \in \mathcal{B}_{k,m}$, the term $\partial_x f(\tilde{\gamma})[u,\psi]$ is a combination of $\tilde{s}$ terms of the form $f(\gamma)[u,\psi]$ with $\gamma \in \mathcal{B}_{k,m+1}$, so these terms have also the desired form. The reasoning for the second kind of terms is analogous. \par
\ \par
Consequently we obtain that \eqref{P1fNewk} holds true at order $m+1$ for some coefficients $c_{k,m+1} (\gamma)$. \par
\ \par
\noindent
{\bf b.} Let us now prove the estimate \eqref{EstCkm} on these coefficients $c_{k,m+1}(\gamma)$.
Let us consider $\gamma \in \mathcal{B}_{k,m+1}$. According to the analysis above, a term of the form $f(\gamma) [u,\psi]$ in $F^{k,m+1}[u,\psi]$ may be originated from one of the three terms in the right hand side of \eqref{Comm}. We estimate the contribution to the coefficient $c_{k,m+1}(\gamma)$ of each of these three types of terms. \par
\begin{itemize}
\item[$\bullet$] {\it Terms of the first kind.} 
Let us assume that in the process described above, the term corresponding to $\gamma=(s,\alpha,\beta) \in \mathcal{B}_{k,m+1}$ was created as a term of the form $\partial_x D^{\tilde{\alpha}_1} u \cdot \ldots \cdot \partial_x D^{\tilde{\alpha}_{\tilde{s}-1 } } u \cdot  \partial_x^{m+1}   D^{ \tilde{\alpha}_{\tilde{s}} } \psi $ coming from $F^k [u,  \partial_x^{m} \psi ]$, for some $\tilde{\theta}= (\tilde{s},\tilde{\alpha}) \in {\mathcal A}_{k}$
. \par
Since this term is of the form $f(\gamma)[u,\psi]$ we have that $\tilde{s}=s$, $\tilde{\alpha} =\alpha$ and ${\beta} =(1, \ldots, 1, m+1) $. 
Therefore the contribution of this kind of terms to $c_{k,m+1}(\gamma)$ can be estimated by
\begin{equation} \label{estiNa}
N_{a} :=  \frac{ k!}{{\alpha} ! } \leq \frac{ k! (m+1) ! }{{\alpha} ! {\beta} !} \leq {(2s)}^{2(m-1)} \, \frac{ k! \, (m+1) !}{\alpha! \, \beta! } .
\end{equation}
\item[$\bullet$] {\it Terms of the second kind.} 
Now let us assume that the term corresponding to $\gamma=(s,\alpha,\beta)$ is originated as a second term in the right hand side of \eqref{Comm}, that is, it can be written as a term in
\begin{equation} \label{tare}
\partial_x D^{\tilde{\alpha}_1} u \cdot \ldots \cdot \partial_x D^{\tilde{\alpha}_{ \tilde{s}-1 } } u \cdot \partial_x f(\tilde{\gamma}) [u,\psi] ,
\end{equation}
for some $\tilde{\theta}=(\tilde{s},\tilde{\alpha}) \in {\mathcal A}_{k}$ and $\tilde{\gamma} =(j,\underline{\alpha},\tilde{\beta}) \in {\mathcal B}_{\tilde{\alpha}_{\tilde{s}},m}$. \par
Now we see that since this term is of the form $f(\gamma)[u,\psi]$, the indices $s$, $\alpha$, $\beta$ and $\tilde{s}$, $\tilde{\alpha}$, $\tilde{\beta}$, $\tilde{\gamma}$ are connected as follows:
\begin{itemize}
\item one has  $\tilde{s} - 1 + j =s$,
\item one has $\alpha = (\tilde{\alpha}_{1}, \ldots, \tilde{\alpha}_{\tilde{s} - 1}, \underline{\alpha})$,
\item the multi-index $\beta$ has the form ${\beta} := (1,\ldots, 1,\underline{\beta} )$, for some $\underline{\beta} \in \N^j$, 
\item moreover $\underline{\beta}$ equals to $\tilde{\beta}$ up to a component for which one adds $1$ to $\tilde{\beta}_{i}$. In other words, if one denotes for $r \in \{ 1, \ldots , j \}$, 
\begin{equation*}
{\mathcal R}_{r}(\underline{\beta}) := (\underline{\beta}_{1}, \ldots, \underline{\beta}_{r-1}, \underline{\beta}_{r}-1,\underline{\beta}_{r+1}, \ldots, \underline{\beta}_{j} ),
\end{equation*}
then for some $r \in \{ 1, \ldots , j \}$, one has $\tilde{\beta} = {\mathcal R}_{r}(\underline{\beta})$.
\end{itemize}
Now given $\gamma$ and $j$, there is a unique $(\tilde{\alpha}, \underline{\alpha})$
and at most $j$ different possibilities of $\tilde{\beta}$ such that this process can generate at $f(\gamma)[u,\psi]$ term (according to $r$). Moreover such a term \eqref{tare} comes with a coefficient $c_{k}(\tilde{\theta}) \, c_{\tilde{\alpha}_{\tilde{s}},m}(\tilde{\gamma})$. It follows that the contribution of these terms to the coefficient $c_{k,m+1}(\gamma)$ can be estimated by
\begin{equation*}
N_{b} := \sum_{j=2}^{s} \sum_{r=1}^{j} \, c_k (\tilde{s} , \tilde{\alpha} )\,  c_{\tilde{\alpha}_{\tilde{s}} ,m} (\tilde{\gamma}) .
\end{equation*}
Using the previous iterations, Lemma \ref{P1} and $\tilde{\beta} = {\mathcal R}_{r}(\underline{\beta})$ we deduce
\begin{align*}
N_{b} &\leq \sum_{j=2}^{s}  \sum_{r=1}^{j} (2j)^{2(m-1)}  \frac{k ! }{\alpha_{1}  ! \ldots \alpha_{\tilde{s}-1} ! \, \tilde{\alpha}_{\tilde{s}} !} \frac{\tilde{\alpha}_{\tilde{s}} ! \, m!}{\underline{\alpha} ! \, \tilde{\beta}  ! } \\
&\leq \sum_{j=2}^{s}  \sum_{r=1}^{j} (2j)^{2(m-1)}  \frac{k ! }{\alpha_{1}  ! \ldots \alpha_{\tilde{s}-1} !} \frac{m!}{\underline{\alpha} ! \, \underline{\beta}  ! } \underline{\beta}_{r}.
\end{align*}
With $\alpha! := \alpha_{1}! \ldots \alpha_{\tilde{s}-1}! \, \underline{\alpha}!$ and $\underline{\beta}! = {\beta}!$ we infer
\begin{equation*}
N_{b} \leq \sum_{j=2}^{s} (2j)^{2(m-1)} \frac{k! \, m!}{\alpha! \, \beta!} |\underline{\beta}| .
\end{equation*}
Now using  $|\underline{\beta}| \leq |\beta|=m+s$, we obtain
\begin{equation} \label{estiNb}
N_{b} \leq (m+s) \sum_{j=2}^{s} (2j)^{2(m-1)} \frac{k! \, m!}{\alpha! \, \beta!}
\leq s(m+s) (2s)^{2(m-1)} \frac{k! \, m!}{\alpha! \, \beta! } .
\end{equation}
\item[$\bullet$] {\it Terms of the third kind.} 
Now assume that the term corresponding to $\gamma$ comes from $ \partial_x F^{k,m} [u,\psi ] $. In that case the contribution to the coefficient $c_{k,m+1}(\gamma)$ can be estimated by
\begin{equation*}
N_{c} := \sum_{j=1}^{s} | c_{k,m} ( \tilde{\gamma}_{j} ) | 
\end{equation*}
such terms, where $\tilde{\gamma}_{j} := (s, \alpha, \tilde{\beta}_{j})$ with $\tilde{\beta}_{j} := {\mathcal R}_{j}(\beta)= ({\beta}_{1} ,\ldots , {\beta}_{j-1} , \beta_{j} - 1, {\beta}_{j+1}, \ldots , \beta_{s})$ (assuming $\beta_{j} \geq 1$).
By the induction hypothesis, we get 
\begin{equation} \label{estiNc}
N_{c}  \leq  \sum_{j=1}^{s} (2j)^{2(m-1)} \frac{ k! \, m!}{\alpha! \, \beta! } \beta_{j} \leq  (2s)^{2(m-1)} \frac{k! \, m!}{\alpha! \, \beta!}  |\beta| 
= (2s)^{2(m-1)} \frac{k! \, m!}{\alpha! \, \beta! } (m+s).
\end{equation}
\end{itemize}
Finally, gathering \eqref{estiNa}-\eqref{estiNc} we obtain
\begin{eqnarray*}
 |c_{k,m+1} ( \gamma) | \leqslant  (2s)^{2(m-1)} \frac{ k! m!  }{\alpha !  \beta !} (m + 1 + (s+1) (m +s)  ) 
 \leqslant  (2s)^{2m} \frac{ k! (m+1)!  }{\alpha !  \beta !}  ,
\end{eqnarray*}
which concludes the proof of Lemma \ref{P1k}.
\end{proof}
\ \par
\noindent
{\bf Propagation of inequalities.}
Using the above formal identity, one can prove the following result.
\begin{Lemma} \label{Cocoa}
There exists $\gamma: \R^{+} \rightarrow \R^{+}$ a positive decreasing function satisfying $\displaystyle {\lim_{L \rightarrow + \infty}} \gamma(L)=0$ such that for any $k \in \N$, 
if  for any $j,m \in \N$ such that $0\leq  m \leq 2l -1$ and $0\leq j \leq k -1$ one has
\begin{equation} \label{GDhyp}
 \|  \partial_{x}^{m}  D^{j} u \|_{  L^{\infty} ((-\tau,\tau) \times \T)} \leq  4   \V_{m,j} ,
\end{equation}
with 
\begin{equation} \label{DefV}
{\V}_{m,j} := \frac{m! j! }{(m+1)^{2} (j+1)^{2}} L^{j} {\V}^{j+1}
\ \text{ and } \ 
{\V} := \|u\|_{ L^\infty  ( -\tau,\tau;   W^{2l-1,\infty} (\T))},
\end{equation}
then, for any $m \in \N$ such that $0\leq  m \leq 2l -1$,
\begin{equation} \label{GDcl}
 \|  F^{k,m} [u,u]  \|_{  L^{\infty} ((-\tau,\tau) \times \T)} \leq   \gamma (L)  {\V}_{m,k} .
\end{equation}
\end{Lemma}
\begin{proof}[Proof of Lemma \ref{Cocoa}]
Here we simply denote $\| \cdot  \|$ for $\| \cdot  \|_{L^{\infty}((-\tau,\tau) \times \T)}$.
First, for any $\gamma := (s,\alpha,\beta) \in  \mathcal{B}_{k,m}$, one has, using \eqref{GDhyp},
\begin{equation} \nonumber
 \|  f(\gamma)  [u,u]  \| \leq   4^{s} \prod_{i=1}^{s} \V_{\beta_{i} , \alpha_{i}} 
=  4^{s}\alpha ! \beta !  \Big( \prod_{i=1}^{s}  \frac{1 }{(\alpha_{i}+1)^{2} (\beta_{i}+1)^{2}} \Big) L^{k+1-s} {\V}^{k+1} .
 \end{equation}
Denoting for $\alpha \in \N^{s}$, 
\begin{equation*}
\Upsilon(s,\alpha):=\prod_{i=1}^s \frac{1}{(1+\alpha_{i})^2},
\end{equation*}
we therefore obtain, using \eqref{EstCkm},
\begin{eqnarray*}
\|  F^{k,m} [u,u]  \| &\leq & \sum_{\gamma    \in \mathcal{B}_{k,m}  } |c_{k,m} (\gamma   )|  \,  \| f(\gamma)  [u,u] \|, \\ 
&\leq &  k! \, m! \, {\V}^{k+1} \sum_{s= 2}^{k+1}\,  {(2s)}^{2(m-1)}   4^{s}  L^{k+1-s} 
 \Big( \sum_{ \alpha \, / \, | \alpha | = k+ 1 -s } \,  \Upsilon(s,\alpha)  \Big)
 \Big(  \sum_{ \beta \, / \, | \beta | = m+ s - 1} \, \Upsilon(s,\beta)   \Big) .
\end{eqnarray*}
We now use \cite[Lemma 7.3.3]{cheminsmf}, which we recall for the reader's convenience.  
\begin{Lemma}[\cite{cheminsmf}] \label{LemmeCheminSMF}
For any couple of positive integers $(s,m)$ we have
\begin{equation*}
\sum_{\substack{{\alpha \in \N^{s}} \\ {|\alpha|=m} }}  \Upsilon(s,\alpha) \leq \frac{20^{s}}{(m+1)^{2}} .
\end{equation*}
\end{Lemma}
\noindent
This yields \eqref{GDcl} with 
\begin{equation*}
\gamma (L)  := \sup_{k}   \sum_{s= 2}^{k+1}\,  {(2s)}^{2(m-1)}   1600^{s}  L^{1-s} 
 \frac{(k+1)^{2} (m+1)^{2}}{( k+ 2 -s )^{2} ( m + s )^{2}}.
\end{equation*} 
It is straightforward to see that $\gamma(L) \rightarrow 0$ as $L \rightarrow +\infty$.
\end{proof}
In the sequel we will use the same notation $\gamma : \R^{+} \rightarrow \R^{+}$ for some other positive decreasing functions such that  $\displaystyle {\lim_{L \rightarrow + \infty}} \gamma(L)=0$, which may change from line to line. Yet these functions will always be, as above, independent of $k$. \par
In particular with some slight modifications left to the reader we also get:
\begin{Lemma} \label{Cacao}
 There exists $\gamma$ a positive decreasing function, independent of $k$, with $\displaystyle {\lim_{L \rightarrow + \infty}} \gamma(L)=0$ such that for any $k \in \N$, 
if  for any $j,m \in \N$ such that $0\leq  m \leq 2l -1$ and $0\leq j \leq k -1$,
\begin{equation*}
\|  \partial_{x}^{m}  D^{j} u \|_{  L^{\infty}((-\tau,\tau) \times \T) } \leq  4   \V_{m,j} 
\ \text{ and } \ 
\|  \partial_{x}^{m} D^{j}  {\Lambda}_{\pm}   P \|_{  L^{\infty}((-\tau,\tau) \times \T)} \leq L \V  \V_{m,j}
\end{equation*}
then,  for any $m \in \N$ such that $0\leq  m \leq 2l -1$,
\begin{equation*}
 \|  F^{k,m} [u,{\Lambda}_{\pm}   P]  \|_{  L^{\infty}((-\tau,\tau) \times \T) } \leq   \gamma (L)  L \V {\V}_{m,k} .
\end{equation*}
\end{Lemma}

Let us also provide the following technical lemma relying on Leibniz's rule. 
\begin{Lemma} \label{Leib}
For any $m_1 , m_2$ in $\N$, for any $k \in \N$, 
if  for any $j \in \N$ such that  $0\leq j \leq k $, for $i=1,2$, one has
\begin{equation*}
\| D^{j} f_i \|_{  L^{\infty}((-\tau,\tau) \times \T)} \leq  5 \V_{m_i ,j} ,
\end{equation*}
then 
\begin{equation*}
\Big\| \sum_{j=0}^{k} \binom{k}{j}   (D^{j} f_{1} ) (D^{k-j}  f_{2} ) \Big\|_{ L^{\infty}((-\tau,\tau) \times \T)} 
\leq  400  {\V} {\V}_{m_{1} + m_{2},k} .
\end{equation*}
\end{Lemma}
\begin{proof}[Proof of Lemma \ref{Leib}]
It follows straightforwardly from  the following inequality:
\begin{eqnarray*}
\sum_{j=0}^{k} \binom{k}{j} {\V}_{m_{1} ,j} {\V}_{m_{2},k-j}
&=&  L^{k} \V^{k+2} k! \, \frac{m_{1}! \, m_{2}!}{(m_{1}+1)^{2} (m_{2}+1)^{2}} \, \sum_{j=0}^{k} \frac{1}{(j+1)^{2}(k-j+1)^{2}} \\
&\leq&  2 L^{k} \V^{k+2} k! \, \frac{(m_{1} + m_{2})!}{(m_{1}+m_{2}+1)^{2} } \, \sum_{j=0}^{\lfloor k/2 \rfloor +1} \frac{1}{(j+1)^{2}(k-j+1)^{2}} \\
&\leq& 2 L^{k} \V^{k+2} k! \, \frac{(m_{1} + m_{2})!}{(m_{1}+m_{2}+1)^{2} } \, \frac{\pi^{2}}{6} \, \frac{4}{k^{2}}  \\
&\leq& 16 {\V}  {\V}_{m_{1} + m_{2} , k} .
\end{eqnarray*}
\end{proof}
\subsection{Core of the proof}
Let us consider $\tau \in (0,T)$. \par
We first deal with smooth solutions $u$ of class $C^{\infty}$. We will explain at the end of this section how to extend the analysis to the solutions tackled in Theorem \ref{start2k}. \par
First we obtain the following a priori estimates.
\begin{Proposition} \label{recuk}
There exists $L>0$ such that the following holds true. Let $(u,P)$ be some smooth functions satisfying the equation \eqref{CHk} on $(-\tau,\tau)$. 
Then for any $k \in \N$, for any $j \in \N$ such that $0\leq m\leq 2l -1$, 
\begin{eqnarray} \label{induk}
 \|  \partial_{x}^{m} D^{k}  {\Lambda}_{\pm}   P \|_{L^{\infty}((-\tau,\tau) \times \T)} \leq L \V  \V_{m,k}, \\
\label{induTk}
 \|  \partial_{x}^{m}  D^{k} u \|_{L^{\infty} ((-\tau,\tau) \times \T)} \leq 4 \V_{m,k} .
\end{eqnarray}
\end{Proposition}
Let us recall here that the operators ${\Lambda}_{\pm} $ are defined in \eqref{deflambda} and $\V$, $\V_{m,k}$ are defined in \eqref{DefV}.
\begin{proof}[Proof of Proposition \ref{recuk}]
We establish \eqref{induk}-\eqref{induTk} by induction on $k$. \par
\ \par
\noindent
{\bf 1.} For $k=0$, \eqref{induTk} is straightforward. To get \eqref{induk}, we observe that
\begin{equation*}
- \Lambda_{+} P = \tilde{\Lambda}_{+}^{-1} \mathcal{F} [u]
\ \text{ and } \
- \Lambda_{-} P =  \tilde{\Lambda}_{-}^{-1}  \mathcal{F} [u].
\end{equation*}
Now it follows from \eqref{CHkF} and Lemma \ref{LemElliptic2l-1} that \eqref{induk} at rank $0$ holds true for $L$ large enough. \par
\ \par
\noindent
{\bf 2.} Let us now consider $k >0$. We assume that \eqref{induk} and \eqref{induTk} hold true up to the order $k-1$, and aim at proving them at rank $k$, for $L$ large enough independent of $k$.  \par
\ \par
\noindent
{\bf a.} First from \eqref{CHk} and \eqref{secon} we infer that 
\begin{equation*}
-2i D^k u = - D^{k-1} ( \Lambda_{+} + \Lambda_{-} ) P .
\end{equation*}
Together with the induction hypothesis, we deduce \eqref{induTk} holds true at rank $k$.
It remains to prove that \eqref{induk} holds true at rank $k$ as well. \par
\ \par
\noindent
{\bf b.} We begin by estimating $D^k  \mathcal{F} [u] $ in  $L^\infty (( -\tau,\tau) \times \T)$.
Applying $D^k$ to  \eqref{CHkF} and using Leibniz's rule, we get 
\begin{equation*}
D^k  \mathcal{F} [u] =  \sum_{0 \leq m_{1} + m_{2} \leq 2l - 1 } \mathfrak{c}_{m_{1}  , m_{2} } 
 \sum_{j=0}^{k} \binom{k}{j} ( D^{j} \partial_x^{m_{1}} u  ) (D^{k-j}  \partial_x^{m_{2}} u ) .
\end{equation*}
Thanks to Lemma \ref{P1k}, we deduce that
\begin{equation*}
D^{j} \partial_x^{m_{1}} u = \partial_x^{m_{1}} D^{j} u  - F^{j,m_{1}} [u, u] 
\ \text{ and } \ 
D^{k-j}  \partial_x^{m_{2}} u =  \partial_x^{m_{2}} D^{k-j}   u - F^{k-j,m_{2}} [u, u] .
\end{equation*}
Then using Lemma \ref{Cocoa} and the fact that \eqref{induTk} is true up to rank $k$, we deduce that for $L$ large enough, 
\begin{equation*}
\|   D^{j} \partial_x^{m_{1}} u \|_{  L^{\infty} ( (-\tau,\tau) \times \T)} \leq   5{\V}_{m_{1},j} 
\ \text{ and } \ 
\| D^{k-j}  \partial_x^{m_{2}} u   \|_{  L^{\infty} ( (-\tau,\tau) \times \T)} \leq  5{\V}_{m_{2},k-j}   .
\end{equation*}
Now using Lemma \ref{Leib} we deduce 
\begin{equation*}
\Big\|  \sum_{j=0}^{k} \binom{k}{j} ( D^{j} \partial_x^{m_{1}} u  ) (D^{k-j}  \partial_x^{m_{2}} u ) \Big\|_{  L^{\infty} ( (-\tau,\tau) \times \T)} 
\leq  400  {\V} {\V}_{m_{1} + m_{2},k} .
\end{equation*}
Now we observe that $l$ being fixed, there exists some constants $C_{1} , C_{2} > 0$ such that for any $m,k \in \N$ with $0 \leq m \leq 2l-1$, 
\begin{equation*}
C_{1} {\V}_{2l-1,k}  \leq {\V}_{m,k} \leq C_{2} {\V}_{2l-1,k}.
\end{equation*}
It follows that
\begin{equation} \label{indu3}
\|  D^{k}   \mathcal{F} [u]  \|_{L^\infty ( (-\tau,\tau) \times \T)} \leq  C  {\V} {\V}_{2l-1,k} .
\end{equation}
\ \par
\noindent
{\bf c.} Next we estimate $\tilde{\Lambda}_{\pm} D^k {\Lambda}_{\pm} P $  in $L^\infty ((-\tau,\tau) \times \T)$. Applying $D^k$ to the second equation in   \eqref{CHk} and using the relations in \eqref{secon}, we dget
\begin{equation*}
D^{k} A P =  D^{k} \tilde{\Lambda}_{\pm} \Lambda_{\pm} P = D^{k} {\mathcal F}[u].
\end{equation*}
Now, with the identity \eqref{combibi} and Lemma \ref{P1k} we deduce 
\begin{equation} \label{Step4a}
 \tilde{\Lambda}_{\pm} D^k  {\Lambda}_{\pm} P =   D^k   \mathcal{F} [u] + 
\sum_{0 \leq  m \leq 2l-1  } d_{m} F^{k,m} [u, {\Lambda}_{\pm} P] .
\end{equation}
The first term in  the right hand side of \eqref{Step4a} was estimated in \eqref{indu3}.
To estimate the second one we use Lemma \ref{Cacao} and modify again $\gamma$ to get 
%
\begin{equation} \nonumber
\|  F^{k,m}  [u, {\Lambda}_{\pm} P ] \|_{L^\infty ((-\tau,\tau) \times \T)} \leq \gamma(L) L  {\V}  {\V}_{m,k}     .
\end{equation}
Thus, modifying again $\gamma(L)$, we get
\begin{equation*}
\| \tilde{\Lambda}_{\pm} D^k  {\Lambda}_{\pm} P  \|_{L^\infty ((-\tau,\tau) \times \T)} \leq \gamma (L) L  {\V}  {\V}_{2l-1,k} .
\end{equation*}
\ \par
\noindent
{\bf d.} Finally we use Lemma \ref{LemElliptic2l-1}, and deduce that
\begin{equation*}
\| D^k  {\Lambda}_{\pm} P  \|_{L^\infty (-\tau,\tau; W^{2l-1}(\T))} \leq C \gamma (L) L  {\V}  {\V}_{2l-1,k} .
\end{equation*}

We choose $L$ large enough to absorb the constant $C$ (independently of $k$) and deduce that \eqref{induk} holds up to order $k$. 
The proof of Proposition \ref{recuk} is over.
\end{proof}
\ \par
\noindent
\begin{proof}[End of the proof of Theorem \ref{start2}]
We still suppose in a first time that $u$ is smooth. Now from the definition of $\xi$ we infer that 
\begin{equation*}
\partial^{k +1 }_t \xi (t,x)= D^k u (t, \xi (t,x)).
\end{equation*}
As mentioned above, the proof of Theorem \ref{start1k} provides that the flow
$\xi$ is $C^{1}$ over $( -\tau,\tau) $ with values in $ W^{2l-1,\infty}( \T) $. 
Thus we infer the following estimates for the flow map: there exists $L>0$ such that for any $k \in \N$,
\begin{equation} \label{indun}
\|  \partial^{k +1 }_t \xi   \|_{L^\infty ( -\tau,\tau; W^{2l-1,\infty}(\T) )} 
\leq \frac{k! L^{k}}{(k+1)^{2}} \| u  \|^{k+1}_{L^\infty ( -\tau,\tau; W^{2l-1,\infty}(\T))}.
\end{equation}
(One can again choose $L$ large enough to absorb the constants depending on $m$.)
In particular this yields the analyticity in time of $\xi$ in the case where $u$ is a $C^{\infty}$ solution. \par
Let us finally briefly explain how to deduce the result for the general flows considered in Theorem \ref{start2k}.
Let $u_0$ be in $W^{2l-1,\infty} (\T)$. By convolution there exists a sequence of  smooth functions $(u_0^{n} )_{n \in \N}$ which satisfies $ \| u_0^{n}   \|_{W^{2l-1,\infty} ( \T) } \leq  \| u_0 \|_{W^{2l-1,\infty} ( \T) }$
and converges to $u_0$ in $H^{2l-1} (\T)$. \par
To these initial data $u_{0}^{n}$ one associates the corresponding smooth flows $(\xi_n )_{n \in \N}$ for which the previous analysis can be applied. In particular the estimates \eqref{indun} hold true when one substitutes $\xi_n $ instead of $\xi$ in the right hand side. But due to the convergence of $u^{n}_{0}$ to $u_{0}$ in $H^{2l-1}(\T)$, we deduce with \eqref{DCk} that the corresponding solutions $u^{n}$ converge to the limit solution $u$ in $C([-\tau,\tau], W^{2l-1,\infty}(\T))$. It follows that the flows $(\xi_n )_{n \in \N}$ converge uniformly to $\xi$ (locally in time). Hence passing to the limit when $n \rightarrow + \infty$ we obtain that $\xi$ satisfies the estimates \eqref{indun} as well, and consequently belongs to the space $C^{\omega } (  (-T,T), W^{2l-1,\infty}(\T))$ of the analytic functions from $(-T,T)$ to $W^{2l-1,\infty} ( \T)$. This concludes the proof of Theorem \ref{start2}. 
\end{proof} 
\section{A concluding remark}
\label{picon}
The Camassa-Holm equation possesses solutions of a soliton type, which, because of their shape (at the peak, the derivative is discontinuous), have been given the name of peakons.
When the equation is set on the whole line, a single peakon is given by
\begin{equation} \label{peakon}
 u_c (t,x) = c e^{-|x-ct|} ,
\end{equation}
with $c >0$, whereas it is called an antipeakon when $c<0$.
The traveling speed is then equal to the height of the peak. By taking a linear combination of peakons one obtains what is called a multipeakon solution. 
\begin{equation} \label{multipeakon}
 u (t,x) =  \sum_{i=1}^{n} p_{i} (t) e^{-|x- q_{i} (t)  | }.
\end{equation}
Then the positions $q(t) := (q_{i }(t) )_{1 \leqslant i \leqslant n}$ and the amplitudes $p(t) := (p_{i} (t) )_{1 \leqslant i \leqslant n}$ of the crests are determined in a non-linear fashion by the interaction.
Actually  plugging the ansatz \eqref{multipeakon} into Eq. \eqref{CH}  we obtain 
the following  system of ordinary differential equations:
\begin{equation} \label{hamilton}
\dot{q}_{i} = \sum_{m=1}^{n} p_{j } e^{-|q_{i } - q_{j } | } , \quad \dot{p}_{i} = \frac{1}{2} \sum_{j=1}^{n} p_{i  } p_{j } \text{sign} (q_{i } - q_{j } ) e^{-|q_{i } - q_{j } | } .
\end{equation}
When $u$ is given by the ansatz \eqref{multipeakon}, then $m$, defined in \eqref{m}, is equal to the following combination of Dirac masses:
\begin{equation}  \label{mmultipeakon}
 m =  \sum_{i=1}^{n} 2 p_{i} (t) \delta_{ q_{i} (t)  }
\end{equation}
and therefore the energy defined in \eqref{pe} is given by
\begin{equation} \label{pem}
\|  u\|^{2}_{H^1 (\R) } = 2 \sum_{1  \leq i,j \leq n} p_{i} (t) p_{j} (t) e^{-|q_{i } (t) - q_{j }(t)  | }  .
\end{equation}
Since the energy is time-independant we obtain that the  system  \eqref{hamilton} reads
\begin{equation}  \label{hamilton2}
\dot{q}_{i} = \partial_{p_{i} } \mathcal{H} , \quad \dot{p}_{i} =  -  \partial_{q_{i} } \mathcal{H} ,
\end{equation}
with Hamiltonian
\begin{equation} \label{hamilton3}
 \mathcal{H} (p,q) :=  \frac{1}{2} \sum_{1  \leq i,j \leq n} p_{i} (t) p_{j} (t) e^{-|q_{i }(t) - q_{j }(t) | }   ,
\end{equation}
It is straightforward that solutions of  \eqref{hamilton} are analytic with respect to time up to the first collision. \par
Higher peakons move faster than the smaller ones, and when a higher peakon overtakes a smaller, there is an exchange of mass, but no wave breaking takes place, in the sense that the solution remains Lipschitz. Furthermore, the $q_{i} (t)$ remain distinct.  However, if some of the $p_{i} (0)$ have some opposite signs, for instance when one peakon and one antipeakon meet, wave breaking may occur. \par
The problem of continuation beyond wave breaking was recently considered by Bressan and Constantin \cite{BC1} improving earlier results on the  $H^1$ theory (see the introduction of  \cite{BC1} for more about the history of this theory). They reformulated the Camassa-Holm equation as a system of ordinary differential equations taking values in a Banach space. This formulation allowed them to continue the solution beyond the collision time, giving a global conservative solution for which  the energy is conserved for almost all times. 
 Holden and Reynaud proceed in the same way but with a different set of variables which simply corresponds to the transformation between Eulerian and Lagrangian coordinates see \cite{HR1,HR2}.
 In the context of peakon-antipeakon collisions they considered the solution where the peakons and antipeakons ``passed through each other''. Local existence of the ODE system is obtained by a contraction argument. Furthermore, their clever reformulation allows for a global solution where all singularities disappear. Going back to the original function $u$, one obtains a global solution of the Camassa-Holm equation.  In addition to the solution u, one includes a family of non-negative Radon measures 
$\mu_{t}$ with density $(u+u^{2}_{x}) dx$ with respect to the Lebesgue measure. The pair $(u, \mu_t)$ is associated to a continuous semigroup and in particular, one has uniqueness and stability.
 Very recently, Bressan and Fonte \cite{BF}  constructed a semi-group of conservative solutions for the Camassa-Holm equation which is continuous with respect to a distance $J$, defined in term of a optimal transportation problem, which is intermediate between $H^1$ and $L^{1}$. The semi-group is actually constructed as the uniform limit of multi-peakon solutions. \par
The investigation of the explicit example of the collision of a peakon and of an antisymmetric antipeakon reveals that the analyticity of the trajectories of the crests still holds even through the collision if one selects the conservative scenario after the collision with an appropriate labelling. Indeed if one look at the collision at $(t,x) = (0,0)$ of a peakon   $p_{1} (t) e^{-|x- q_{1} (t)  | }$ with $p_{1} (t) ,q'_{1} (t)  \rightarrow c_{1} $ when $t \rightarrow -\infty$ and of an antipeakon  $p_{2} (t) e^{-|x- q_{2} (t)  | }$ with $p_{2} (t) ,q'_{2} (t)  \rightarrow c_{2} $ when $t \rightarrow -\infty$,   for some $c_{1} > 0$, $c_{2} < 0$, then one has to glue together the trajectories
\begin{equation*}
q_{1} (t) =  \ln \left( \frac{c_{1} - c_{2}}{  c_{1} e^{-c_{1} t} - c_{2}e^{-c_{2} t} } \right), \ \ 
q_{2 } (t) = - \ln \left( \frac{c_{1} - c_{2}}{  c_{1} e^{c_{1} t} - c_{2}e^{c_{2} t} } \right) ,
\end{equation*}
for $t<0$ with some trajectories of the form 
\begin{equation*}
\tilde{q}_{1} (t) = \ln \left( \frac{\tilde{c}_{1} - \tilde{c}_{2}}{  \tilde{c}_{1} e^{-\tilde{c}_{1} t} - \tilde{c}_{2}e^{-\tilde{c}_{2} t} } \right) , \ \
\tilde{q}_{2} (t) = - \ln \left(\frac{\tilde{c}_{1} - \tilde{c}_{2}}{  \tilde{c}_{1} e^{\tilde{c}_{1} t} - \tilde{c}_{2}e^{\tilde{c}_{2} t} } \right) ,
\end{equation*}
for $t>0$. 
If one looks for the solutions which verify that  $ \int_{ \R} u(t,y)  dy $ and $\|  u(t,\cdot)  \|_{H^1 (\R) }$ are time independent, then one must have $\{ \tilde{c}_{1} ,  \tilde{c}_{2}\} = \{ {c}_{1} ,  {c}_{2}\}$, that is, with the arbitrary choice that $ \tilde{c}_{1} < \tilde{c}_{2}$,  $ \tilde{c}_{1} = c_{1}$, $ \tilde{c}_{2} = {c}_{2} $. 
It then remains to choose if one prolongs $(q_{1} (t), q_{2} (t))_{t<0}$ by $(\tilde{q}_{1} (t), \tilde{q}_{2} (t))_{t>0}$ or $(\tilde{q}_{2} (t), \tilde{q}_{1} (t) )_{t>0}$. 
Yet since $\lim_{t \rightarrow 0^{-}} q''_{1} (t) = c_{1} c_{2}$ and $\lim_{t \rightarrow 0^{+}} \tilde{q}''_{2} (t) = - c_{1} c_{2}$, the only choice which yields analytic trajectories for $t$ running over the full line is the first one. \par
It should be therefore interesting to determine whether Theorem \ref{start2} could be extended to the  $H^1$  conservative theory or not.
Let us stress here that Theorem \ref{start1} contains the case of the periodic counterpart of the peakon given in \eqref{peakon} : the so-called periodic peakons which are of the form 
\begin{equation} \label{ppeakon}
 u_\gamma (t,x) =  \gamma  \sum_{n \in \Z } e^{-|x+n -ct  | } ,
\end{equation}
where $\gamma $ is chosen in such a way that the height of the crest of the peakon  is the same as the speed of the crest, that is such that $u_\gamma (t,ct) = c$ (cf. \cite[Section 2.1]{suisses}). \par
\section*{Appendix. Proof of Theorem \ref{start1k}}
In this appendix we give a sketch of proof of Theorem \ref{start1k}. We will follow closely the method of  \cite{suisses}.
\begin{Definition}[] \label{}
Let us define $ \mathcal{D}_l$ the Banach manifold given by the union of the following two charts:
\begin{eqnarray*}
 \mathcal{U}_0 := \{ \xi (x) = x + f(x) / \ f \in W^{2l-1,\infty} ( \T) : \ |  f(0)| < 1/2 , \ \text{ essinf  } f'  > -1  \} , \\ 
 \mathcal{U}_l := \{ \xi (x) = x + f(x) / \ f \in W^{2l-1,\infty} ( \T) : \ 0 <  f(0) < 1 , \ \text{ essinf  } f'  > -1   \}.
\end{eqnarray*}
\end{Definition}
Let us introduce here a few notations.
We denote by $R_\xi  v$ the right translation $v \circ \xi$ of $v$ by $ \xi$ and by $P_\xi$ the conjugate  $R_\xi  P  R_{\xi^{-1}} $ of a nonlinear operator $P$. 
We also define the directional derivative of $P_\xi$ in direction $w$ as follows: let $t \mapsto \xi_t$ be a $C^1$-path with values in $ \mathcal{D}_l$  such that $ \xi_0 = \xi$ and $(\frac{d}{dt}  \xi_t )|_{t=0} = w$. Then we define the directional derivative
\begin{equation*}
D_w P_\xi (f) := \frac{d}{dt} |_{t=0}  (P_{\xi_t } f) .
\end{equation*}
We introduce 
\begin{equation*}
f(\xi, v) := - \Big( A^{-1} \partial_x \mathcal{F}_l \Big)_\xi [v]
= - ( A^{-1} \partial_x  )_{\xi}  h , \ \text{ with } \ h := ( \mathcal{F} )_{\xi}[v] .
\end{equation*}
It is sufficient to prove that $f$ is $C^1$ in a neighborhood of $(Id,0)$ in $ \mathcal{D}_l \times W^{2l-1,\infty} ( \T)$ (see \cite[Theorem 2]{suisses}). Then by scaling properties we can infer that the ODE Cauchy problem
\begin{equation*}
(\xi, v)' (t)  = (v, f(\xi, v) ), \ \ (\xi , v) (0) = (Id , u_0 ) ,
\end{equation*}
admits a local-in-time solution in $ \mathcal{D}_l \times W^{2l-1,\infty} ( \T)$ for any $u_0 \in W^{2l-1,\infty} ( \T)$. \par
Since 
\begin{equation*}
A^{-1} \partial_x =2i (  \tilde{\Lambda}_{+} +  \tilde{\Lambda}_{l,-}),
\end{equation*}
this follows from the two next propositions.
\begin{Proposition}
Let $\xi_{0} \in \C \setminus 2 \pi \Z$   and $j \in \N$. Then the mapping
\begin{eqnarray*}
\xi \in  \mathcal{D}_l \mapsto ( -i \partial_x -  \xi_0)_\xi^{-1} \in  \mathcal{L} (W^{j-1,\infty} (\T) ,W^{j,\infty} (\T) )
\end{eqnarray*}
is of class $C^1$.
\end{Proposition}
This can be proved in the same way as \cite[Proposition 3]{suisses}.
\begin{Proposition}
The map $h$ is $C^1$ from $ \mathcal{D}_l \times W^{2l-1,\infty} ( \T)$ to $L^\infty  ( \T)$.
\end{Proposition}
\begin{proof}[Sketch of proof]
The proposition follows easily from the formula:
\begin{equation*}
( \partial_x)_\xi = \frac{1}{\xi'(x)} \partial_x , \quad D_w ( \partial_x)_\xi  = - \frac{w'}{(\xi'(x))^2} \partial_x ,
\end{equation*}
and from the fact that  $ \mathcal{F}_l [u]$ is a differential polynomial in $u$ of order $2l-1$. With these formulas, one deduces that $h(\xi,v)$ is G\^ateaux-differentiable, and that the differential is continuous, so that $h$ is indeed of class $C^{1}$.
\end{proof}
Once the problem is solved in Lagrangian variables one use the following lemma to deduce  the existence part of Theorem \ref{start1k} by considering $u(t,x) := v(t, \xi(t,\cdot)^{-1} (x))$.
\begin{Lemma}
If $ \xi$ is in $ \mathcal{D}_l$ then  $ \xi$ is a homeomorphism and  $ \xi^{-1}$ is in $ \mathcal{D}_l$.
Moreover the following map 
\begin{equation*}
W^{2l-1,\infty} ( \T) \times  \mathcal{D}_l \rightarrow H^{2l-1}(\T) , \ \ \ (u, \xi ) \mapsto u \circ \xi ,
\end{equation*}
is continuous.
\end{Lemma}
The previous lemma simply follows from Fa\`a di Bruno's formula. 
Note that the uniqueness in Eulerian variables  is inferred from the one in  Lagrangian variables, cf. \cite[Proposition 5]{suisses}. \par
\ \par
\noindent
{\bf Acknowledgements.} The  authors are partially supported by the Agence Nationale de la Recherche, Project CISIFS, grant ANR-09-BLAN-0213-02. They wish to thank Rapha\"el Danchin for useful discussions.
\end{document}